\documentclass[12pt]{article} 

\usepackage{amsmath}
\usepackage{amsthm}
\usepackage{amsfonts}
\usepackage{mathrsfs}
\usepackage{stmaryrd}
\usepackage{setspace}
\usepackage{fullpage}
\usepackage{amssymb}
\usepackage{breqn}
\usepackage{enumitem}
\usepackage{bbold} 
\usepackage{authblk}
\usepackage{comment}
\usepackage{hyperref}
\usepackage{pgf,tikz}
\usepackage{graphicx}
\usepackage{subcaption}

\newtheorem*{thm*}{Theorem}
\newtheorem{thm}{Theorem}[section]

\newtheorem{lemma}[thm]{Lemma}

\newtheorem{proposition}[thm]{Proposition}
\newtheorem*{prop*}{Proposition}
\newtheorem{prop}[thm]{Proposition}

\newtheorem{corollary}[thm]{Corollary}

\newtheorem{clm}[thm]{Claim}

\DeclareMathOperator{\Sym}{Sym}
\DeclareMathOperator{\Sep}{Sep}
\DeclareMathOperator{\Age}{Age}
\DeclareMathOperator{\Stab}{Stab}
\DeclareMathOperator{\id}{id}
\DeclareMathOperator{\Inv}{Inv}
\DeclareMathOperator{\Aut}{Aut}
\DeclareMathOperator{\End}{End}
\DeclareMathOperator{\Pol}{Pol}
\DeclareMathOperator{\Proj}{Proj}
\DeclareMathOperator{\Cl}{Cl}
\DeclareMathOperator{\Cll}{Cl_{loc}}
\DeclareMathOperator{\pol}{Pol}
\DeclareMathOperator{\Typ}{Typ}
\DeclareMathOperator{\typ}{Typ}
\DeclareMathOperator{\Type}{Typ}
\DeclareMathOperator{\ran}{ran}
\DeclareMathOperator{\dom}{dom}
\DeclareMathOperator{\cyc}{cycl}
\DeclareMathOperator{\sw}{sw}
\DeclareMathOperator{\sep}{sep}
\DeclareMathOperator{\betw}{betw}
\DeclareMathOperator{\sync}{sync}
\DeclareMathOperator{\link}{link}
\DeclareMathOperator{\wlink}{wlink}
\DeclareMathOperator{\pp}{pp}

\newcommand\ex{\ensuremath{\mathrm{ex}}}
\newcommand\cA{{\mathcal A}}
\newcommand\cB{{\mathcal B}}
\newcommand\cC{{\mathcal C}}
\newcommand\cD{{\mathcal D}}
\newcommand\cE{{\mathcal E}}
\newcommand\cF{{\mathcal F}}
\newcommand\cG{{\mathcal G}}
\newcommand\cH{{\mathcal H}}
\newcommand\cI{{\mathcal I}}
\newcommand\cJ{{\mathcal J}}
\newcommand\cK{{\mathcal K}}
\newcommand\cL{{\mathcal L}}
\newcommand\cM{{\mathcal M}}
\newcommand\cN{{\mathcal N}}
\newcommand\cP{{\mathcal P}}
\newcommand\cY{{\mathcal Y}}
\newcommand\cQ{{\mathcal Q}}
\newcommand\cR{{\mathcal R}}
\newcommand\cS{{\mathcal S}}
\newcommand\cT{{\mathcal T}}
\newcommand\cU{{\mathcal U}}
\newcommand\cV{{\mathcal V}}
\newcommand\PP{{\mathbb P}}
\renewcommand{\P}{\mathbb P}
\renewcommand{\S}{\mathcal S}

\newcommand{\cl}[1]{\langle #1 \rangle}
\newcommand{\To}{\rightarrow}
\newcommand{\nin}{\notin}
\newcommand{\inv}{^{-1}}
\newcommand{\mult}{\times}

\DeclareMathOperator{\tp}{tp}

\DeclareMathOperator{\Rev}{Rev}
\DeclareMathOperator{\Frlim}{Frlim}
\DeclareMathOperator{\R}{Rev}
\DeclareMathOperator{\T}{Turn}
\DeclareMathOperator{\M}{Max}
\DeclareMathOperator{\pari}{Par}

\def\lc{\left\lceil}   
\def\rc{\right\rceil}

\def\lf{\left\lfloor}   
\def\rf{\right\rfloor}

\renewcommand{\qedsymbol}{$\blacksquare$}
\newcommand{\rev}{\updownarrow}
\newcommand{\Puv}{(\P,u,v)}
\newcommand{\oo}{\pari}
\newcommand{\oz}{\cyc}
\newcommand{\od}{\cyc'}
\newcommand{\tf}{\triangleleft_F}
\newcommand{\abs}[1]{\left\lvert{#1}\right\rvert}

\newcommand{\ignore}[1]{}

\newcommand{\clfo}[1]{\langle #1 \rangle_{fo}}

\newcommand{\rest}{\upharpoonright}

\title{On non-degenerate Turán problems for expansions}
\author{Dániel Gerbner\\\small Alfr\'ed R\'enyi Institute of Mathematics\\
\small \texttt{gerbner.daniel@renyi.hu}}

\date{}

\begin{document}

\maketitle

\begin{abstract}
The $r$-uniform expansion $F^{(r)+}$ of a graph $F$ is obtained by enlarging each edge with $r-2$ new vertices such that altogether we use $(r-2)|E(F)|$ new vertices. Two simple lower bounds on the largest number $\ex_r(n,F^{(r)+})$ of $r$-edges in $F^{(r)+}$-free $r$-graphs are $\Omega(n^{r-1})$ (in the case $F$ is not a star) and $\ex(n,K_r,F)$, which is the largest number of $r$-cliques in $n$-vertex $F$-free graphs. We prove that $\ex_r(n,F^{(r)+})=\ex(n,K_r,F)+O(n^{r-1})$. The proof comes with a structure theorem that we use to determine $\ex_r(n,F^{(r)+})$ exactly for some graphs $F$, every $r<\chi(F)$ and sufficiently large $n$.
\end{abstract}




\section{Introduction}

A fundamental theorem in extremal combinatorics is due to Tur\'an \cite{T}, who showed that if an $n$-vertex graph $G$ does not contain $K_{k+1}$ as a subgraph, then $G$ cannot have more edges than the \textit{Tur\'an graph} $T_2(n,k)$, which is the complete balanced $n$-vertex $k$-partite graph. Here \textit{balanced} means that each part is of order either $\lfloor n/k\rfloor$ or $\lceil n/k\rceil$. 

More generally, given a graph $F$, we denote by $\ex(n,F)$ the largest number of edges an $n$-vertex $F$-free graph can have. Let $t_2(n,k)$ denote the number of edges in $T_2(n,k)$.
Erd\H os, Stone and Simonovits \cite{ES1966,ES1946} showed that if a graph $F$ has chromatic number $k+1>2$, then $\ex(n,F)=(1+o(1))t_2(n,k)$. Another result relevant to our studies deals with color-critical edges. An edge of $F$ is said to be \textit{color-critical} if deleting that edge decreases the chromatic number. Simonovits \cite{sim} showed that if $F$ has chromatic number $k+1$ and a color-critical edge, then for sufficiently large $n$ we have $\ex(n,F)=t_2(n,k)$.

Analogous questions can be asked about $r$-uniform hypergraphs ($r$-graphs in short). For an $r$-graph $\cF$ we denote by $\ex_r(n,\cF)$ the largest number of hyperedges in $\cF$-free $n$-vertex $r$-graphs. Let $T_r(n,k)$ denote the complete balanced $k$-partite $r$-graph and $t_r(n,k)$ denote the number of its hyperedges.

Hypergraph Tur\'an problems are notoriously more difficult than the graph versions. Therefore, it is desirable to study forbidden hypergraphs where graph theoretic methods can be applied. This often happens if the hypergraph is built from a graph by enlarging its edges. Chapter 5 of \cite{gp} contains a survey of hypergraph Tur\'an problems with special emphasis on graph-based hypergraphs.

\textit{Expansions} are among the most studied graph-based hypergraphs. Given a graph $F$, its $r$-uniform expansion $F^{(r)+}$ is obtained by adding $r-2$ new vertices to each edge, such that the $|E(F)|(r-2)$ new vertices are distinct from each other and are not in $V(F)$.

There are several early hypergraph results that can be restated in the language of expansions, see the survey \cite{mubver} by Mubayi and Verstra\"ete. It is easy to see (see \cite{mubver}) that $\ex_r(n,F^{(r)+})=\Theta(n^r)$ if and only if $\chi(F)>r$. This case is called \textit{non-degenerate} and is the main subject of the present paper.

Mubayi \cite{mub} showed that if $k\ge r$, then $\ex_r(n,K_{k+1}^{(r)+})=(1+o(1))t_r(n,k)$. He conjectured and Pikhurko \cite{pikhu} proved that for sufficiently large $n$ we have $\ex_r(n,K_{k+1}^{(r)+})=t_r(n,k)$.
The survey \cite{mubver} mentions that Alon and Pikhurko observed that the proof extends to any $(k+1)$-chromatic graph $F$ with a color-critical edge in place of $K_{k+1}$. It is also stated in \cite{mubver} that for any $F$ with $\chi(F)=k+1>r$, $\ex_r(n,F^{(r)+})=(1+o(1))t_r(n,k)$ (they present the key ideas of a proof, for a detailed proof see \cite{pttw}).
We are aware of only one other result on $\ex_r(n,F^{(r)+})$ in the non-degenerate regime (i.e., in the case $\chi(F)>r$), due to Tang, Xu and Yan \cite{txy}.

Given a graph $F$ with $\chi(F)=k+1$, its decomposition family $\cD(F)$ is the family of bipartite graphs obtained by deleting $k-1$ color classes from a proper $(k+1)$-coloring of $F$. Let $\mathrm{biex}(n,F):=\ex(n,\cD(F))$. Tang, Xu and Yan \cite{txy} claim that if $k\ge r$, then $\ex_r(n,F^{(r)+})=t_r(n,k)+\Theta(\mathrm{biex}(n,F)n^{r-2})$. Unfortunately, this is not completely true, there is a case neglected in \cite{txy}. It is claimed there that either $\mathrm{biex}(n,F)=0$, or $\mathrm{biex}(n,F)\ge n/2$. This is true if there is only one forbidden graph in $\cD(F)$, but not in general. Consider the graph $B_{k+1,1}$, which consists of two copies of $K_{k+1}$ sharing exactly one vertex. It is easy to see that $\cD(B_{k+1,1})$ contains both the matching and the star with two edges, thus $\mathrm{biex}(n,B_{k+1,1})=1$. The above-mentioned result of \cite{txy} does not hold for graphs $F$ with $\mathrm{biex}(n,F)=\Theta(1)$ (but holds for each other graph). First, we state a corrected version.

\begin{prop}\label{corre} Let $\chi(F)=k+1>r$ and $\mathrm{biex}(n,F)=\Theta(1)$. Then $\ex_r(n,F^{(r)+})=t_r(n,k)+\Theta(n^{r-1})$.
\end{prop}

Given $p$-graphs $\cH$ and $\cG$, let $\cN(\cH,\cG)$ denote the number of copies of $\cH$ in $\cG$. Given $p$-graphs $\cH$ and $\cF$, we let $\ex_p(n,\cH,\cF)=\max\{\cN(\cH,\cG): \, \text{$\cG$ is an $n$-vertex $\cF$-free $p$-graph}\}$. We omit $p$ from the notation if $p=2$. After several sporadic results, the systematic study of this generalized Tur\'an function was initiated by Alon and Shikhelman \cite{ALS2016} in the graph case. They showed that if $\chi(F)=k+1>r$, then $\ex(n,K_r,F)=(1+o(1))t_r(n,k)$. Ma and Qiu \cite{mq} improved this by showing that $\ex(n,K_r,F)=t_r(n,k)+\Theta(\mathrm{biex}(n,F)n^{r-2})$.

Note that we have the simple bound $\ex_r(n,F^{(r)+})\ge\ex(n,K_r,F)$ for any $r$, $n$ and $F$. Indeed, let us consider an $F$-free $n$-vertex graph $G$ and take the vertex set of every $K_r$ as a hyperedge. The resulting $r$-graph is clearly $F^{(r)+}$-free. If $F$ is not a star, then another simple $F^{(r)+}$-free construction is obtained by taking all the $\binom{n-1}{r-1}$ hyperedges containing a fixed vertex.

We show that the order of magnitude of $\ex_r(n,F^{(r)+})$ is given by the larger of these lower bounds.

\begin{thm}\label{main} For any graph $F$ and integer $r\ge 2$, we have $\ex_r(n,F^{(r)+})=\ex(n,K_r,F)+O(n^{r-1})$.
\end{thm}

In fact, we prove the following more general version.

\begin{thm}\label{main2} For any graph $F$ and integers $r\ge p\ge 2$, we have $\ex_r(n,F^{(r)+})=\ex_p(n,K_r^p,F^{(p)+})+O(n^{r-1})$.
\end{thm}

This immediately gives an improvement in the non-degenerate case if we know $\ex(n,K_r,F)$, and $\mathrm{biex}(n,F)$ is superlinear. Note that we obtain improvements in the degenerate regime as well if $\ex(n,K_r,F)=\omega(n^{r-1})$. For example Alon and Shikhelman \cite{ALS2016} showed that $\ex(n,K_3,K_{s,t})=O(n^{3-3/s})$, and provided a matching lower bound in the case $t>(s-1)!$. Kostochka, Mubayi and Verstrate \cite{kmv} gave the same bounds on $\ex_3(n,K_{s,t}^{(3)+})$. Theorem \ref{main} shows that these results are equivalent, moreover, if $t>(s-1)!>2$, then $\ex(n,K_3,K_{s,t})$ and $\ex_3(n,K_{s,t}^{(3)+})$ not only have the same order of magnitude, but also the same asymptotics.

According to \cite{mubver}, one of the main questions in the topic of expansions is to determine when $\ex_3(n,F^{(3)+})=\Theta(n^2)$. By Theorem \ref{main}, this happens if and only if $\ex(n,K_3,F)=O(n^2)$. 

We can use our methods to obtain some exact results. Let $T_r(n,k,i)$ denote the following $r$-graph. We take a copy of $T_r(n-i,k)$, add $i$ vertices and each $r$-set containing at least one of those $i$ vertices as a hyperedge. 

Let $\cH(m)$ be the complete $k$-partite $n$-vertex $r$-graph with a part $V_1$ of order  $m$ and $k-1$ parts $V_2,\dots,V_k$ each of order either $\lfloor \frac{n-m}{k-1}\rfloor$ or $\lceil \frac{n-m}{k-1}\rceil$. In other words, we take a part of order $m$ and a $(k-1)$-partite Tur\'an hypergraph on the remaining vertices. Let $u,v\in V_1$. Let $\cH'(m)$ denote the $r$-graph obtained from $\cH(m)$ by adding as hyperedges each $r$-set that contains both $u$ and $v$, and each $r$-set that contains $u$ and $r-1$ vertices from parts other than $V_1$. 
Let $\cH_r(n,k)$ denote $\cH'(m)$ with the most hyperedges. Recall that  $B_{k+1,1}$ consists of two copies of $K_{k+1}$ sharing exactly one vertex.


\begin{thm}\label{exac} \textbf{(i)} Let $F$ consist of $s$ components with chromatic number $k+1$, each with
a color-critical edge, and any number of components with chromatic number at most $k$. Then
$\ex_r(n,F^{(r)+})=|E(T_r(n,k,s-1))|$ for sufficiently large $n$.

\textbf{(ii)}
$\ex_r(n,B_{k+1,1}^{(r)+})=|E(\cH_r'(n,k))|$ for sufficiently large $n$.
\end{thm}

Let us discuss what the optimal value of $m$ is.
For given $k$, $r$ and $n$ it is straightforward to determine this value. In general, it is easy to see that if $m>n/k$, then moving a vertex different from $u$ and $v$ from $V_1$ to another part increases the number of hyperedges. Assume now that we move some vertices from one part $V_2$ of the Tur\'an hypergraph to the other parts in a balanced way, one by one. Moving the $i$th vertex $v_i$ to $V_j$, the number of pairs $(w,w')$ with $w\in V_1$, $w'\in V_j$ decreases roughly by $i$, which results in losing roughly $i\binom{k-2}{r-2}\left(\frac{n-m}{k-1}\right)^{r-2}$ hyperedges that intersect each part. On the other hand, we gain $\binom{n-m+i-1}{r-2}$ new hyperedges that contain $u$ and $r-1$ vertices from parts other than $V_1$ (each hyperedge that contains $u$, $v_i$ and $r-2$ vertices from parts other than $V_1$). Therefore, the number of hyperedges increases if $i$ is at most roughly $(k-1)^{r-2}/(k-2)(k-3)\dots (k-r+1)$. 

Let us consider the case $r=3$. Here we are going to be more precise. If $V_1$ has order less than $V_j$, then the number of pairs $(w,w')$ with $w\in V_1$, $w'\in V_j$ decreases by at least two, and this implies that such change does not increase the total number of hyperedges. On the other hand, if $|V_1|=|V_j|$, then it is worth moving a vertex from $V_1$ to $V_j$. 
Therefore, $m=\lfloor (n-1)/(k-1)\rfloor$.


\section{The structure of $F^{(r)+}$-free hypergraphs}



Given an $r$-graph $\cH$, we denote by $\partial\cH$ the \textit{shadow} of $\cH$, which is the $(r-1)$-graph such that an $(r-1)$-set $H$ is a hyperedge of $\partial\cH$ if and only if $H$ is contained in a hyperedge of $\cH$. We say that a set of vertices $A$ is \textit{$t$-heavy} (with respect to $\cH$) if there are at least $t$ hyperedges of $\cH$ that contain $A$. We say that $A$ is \textit{$t$-fat} (with respect to $\cH$) if there are at least $t$ hyperedges of $\cH$ with pairwise intersection equal to $A$. 
Observe that an $(r-1)$-set is $t$-heavy if and only if it is $t$-fat.

Given an $r$-graph $\cH$ and an integer $t$, we let $\partial^*(t)\cH$ denote the $(r-1)$-graph consisting of the $t$-heavy $(r-1)$-sets.

\begin{prop}\label{arny}
Let $\cH$ be an $F^{(r)+}$-free $r$-graph and $t\ge (r-2)|E(F)|+|V(F)|$. Then $\partial^*(t)\cH$ is $F^{(r-1)+}$-free.
\end{prop}

\begin{proof}
Let us assume that there is a copy of $F^{(r-1)+}$ in $\partial^*(t)\cH$ and let $H_1,\dots,H_{|E(F)|}$ be its hyperedges. We go through them greedily and for each $H_i$, we take a hyperedge $H_i'$ of $\cH$ that contains $H_i$. This means we add one more vertex to $H_i$, and we have at least $(r-2)|E(F)|+|V(F)|$ choices. We pick $H_i'$ in such a way that the new vertex is not in the copy of $F^{(r-1)+}$, nor in any $H_j'$ with $j'<i$. As there are less than $(r-2)|E(F)|+|V(F)|$ such vertices, we can always pick $H_i'$ this way. Then the hyperedges $H_i'$ for $i\le |E(F)|$ form a copy of $F^{(r)+}$ in $\cH$, a contradiction.
\end{proof}

   Let $\cG$ be an $n$-vertex $F^{(r)+}$-free $r$-graph and $t= (r-2)|E(F)|+|V(F)|$. We will apply the operator $\partial^*(t)$ repeatedly, and let $\cG_i$ denote the $i$-graph obtained this way for every $i<r$.Then $\cG_i$ is $F^{(i)+}$-free by Proposition \ref{arny}.

Now we are ready to prove Theorem \ref{main2} that we restate here for convenience.

\begin{thm*} For any graph $F$ and integers $r\ge p\ge 2$, we have $\ex_r(n,F^{(r)+})=\ex_p(n,K_r^p,F^{(p)+})+O(n^{r-1})$.
\end{thm*}

\begin{proof}
 We partition $\cG$ into two hypergraphs. Let $\cG'$ denote the hypergraph consisting of the hyperedges $H$ such that every $p$-subset of $H$ is in $\cG_p$, and $\cG''$ denote the remaining hyperedges. Then clearly $|E(\cG')|\le \cN(K_r^p,\cG_p)\le\ex(n,K_r^p,F^{(p)+})$, since each hyperedge is on the vertex set of a $K_r^p$ in $\cG_p$.

    We claim that $|E(\cG'')|=O(n^{r-1})$. By definition of $\partial^*(t)$, for each $r$-edge $H$ in $\cG''$, there is a largest integer $p\le q< r$ such that a $q$-subset $A$ of $H$ is not a hyperedge of $\cG_q$. For any fixed $q$ between $p$ and $r$, there are $O(n^q)$ such $q$-sets. Each such $q$-set is contained in at most $t-1$ hyperedges of $\cG_{q+1}$, which are clearly contained in $O(n^{r-q-1})$ $r$-sets. We counted each hyperedge of $\cG''$ at least once this way, completing the proof.
\end{proof}

Let us focus on $\cG_2$ now. Observe that if $\cG$ has $\ex_r(n,F^{(r)+})-o(n^{r})$ hyperedges, then $\cG_2$ has to contain $\ex(n,K_r,F)-o(n^r)$ copies of $K_r$. Using this, one can show that a stability result of Ma and Qiu \cite{mq} on $\ex(n,K_r,F)$ in the non-degenerate setting is equivalent to a stability result of Tan, Xu and Yan \cite{txy} on $\ex_r(n,F^{(r)+})$.

Ma and Qiu \cite{mq} gave a description of the extremal graphs in the non-degenerate setting, inside a proof. See \cite{ge2} for a more general version that is stated as a separate theorem. It is not hard to see that the proof in \cite{ge2} extends with small modifications to graphs that are close to extremal. We only state it for $K_r$, but the same holds if $K_r$ is replaced by any $H$ with the so-called $F$-Tur\'an-stable property. Given a graph $F$ with $\chi(F)=k+1$, we denote by $\sigma(F)$ the smallest number of vertices that we can delete from $F$ to obtain a graph of chromatic number $k$.

\begin{thm}\label{grafos} Let $\chi(F)=k+1>r$, $\zeta>0$, $n$ sufficiently large and $G$ be an $n$-vertex $F$-free graph G with $\cN(K_r,G) \ge \ex(n,K_r, F)-\zeta n^{r-1}$. Then there is a $k$-partition of $V(G)$ to $V_1,\dots,V_k$, a constant $K=K(F,\zeta)$, a set $B$ of at most
$rK(\sigma(F)-1)$ vertices and a set $U$ of at most $K$ vertices such that each member of $\cD(F)$ inside a part shares at least two
vertices with $B$, every vertex of $B$ is adjacent to 
$\Omega(n)$ vertices in each part, every vertex of $U\cap V_i$ is adjacent to $o(n)$ vertices in $V_i$ and every vertex
of $V_i\setminus (B\cup U)$ is adjacent to $o(n)$ vertices in $V_i$ and all but $o(n)$ vertices in $V_j$ with $j\neq i$. Moreover, if $\zeta$ is sufficiently small, then $U=\emptyset$ for sufficiently large $n$.
\end{thm}

The proof is almost identical to the version in \cite{ge2} (which borrows a lot from \cite{mq}), thus we only give a sketch.

\begin{proof}[Sketch of proof] Let us pick numbers $\alpha,\beta,\zeta,\gamma,\varepsilon>$ in this order, such that each is
sufficiently small compared to the previous one, and after that we pick $n$ that is sufficiently
large. By the stability theorem of Ma and Qiu \cite{mq}, $G$ is close to the Tur\'an graph. More precisely, we can partition the vertex set into $V_1,\dots,V_k$ such that at most $\varepsilon n^2$ edges are missing between parts and $|V_i|\ge \alpha n$ for each $i$.
Let $B_i$
denote the set of vertices in $V_i$ with at least $\gamma n$ neighbors in $V_i$ and $B:=\cup_{i=1}^k B_i$.

The proof of the upper bound on $|B|$ in \cite{ge2} (and in \cite{mq}) does not use the assumption that $\cN(K_r,G)=\ex(n,K_r,F)$, thus it holds in our setting. Let $U_i$ denote the set of vertices in $V_i$ with at least $\beta n$ non-neighbors in some other part $V_j$ and $U:=\cup_{i=1}^k U_i$. In the case $G$ is extremal, it was shown in \cite{mq} that $U$ is empty, by showing that we can change $G$ to another $F$-free graph $G'$ with $\cN(K_r,G')\ge\cN(K_r,G)+|U|\beta n^{r-1}/2^r)$. For us, it just shows that $|U|$ is bounded above by a constant $K$, since $\cN(K_r,G')\le \ex(n,K_r,G)\le \cN(K_r,G)+\zeta n^{r-1}$. If $\zeta$ is sufficiently small compared to $\beta$, then this implies $|U|=0$.

Finally, the assumption on $\cD(F)$ follows as in \cite{ge2}.
\end{proof}

The above theorem together with Theorem \ref{main} implies that the same conclusion holds for $\cG_2$ if $\cG$ is an extremal hypergraph for $\ex_r(n,F^{(r)+})$. We will use $t$-fat edges to get rid of $U$.

\begin{proposition}\label{trivi}
Assume that $t\ge (r-2)|E(F)|+|V(F)|$ and let $F_0$ be a subgraph of $F$.
Assume that we find a copy of $F_0^{(r)+}$ in $\cF$, such that its core $F_0$ is extended to a copy of $F$ with $t$-fat edges in the shadow graph of $\cG$ insuch a way that the only vertices shared by $F_0^{(r)+}$ and $F$ are those in $F_0$. Then this copy is the core of an $F^{(r)+}$ in $\cG$.
\end{proposition}

\begin{proof}
    Consider the  copy of $F$ in the shadow graph containing the core of an $F_0^{(r)+}$. We go through the edges not in the core of $F_0^{(r)+}$ one by one. For each edge $uv$, we pick a hyperedge containing it and avoiding the other vertices in $F_0^{(r)+}$, in $F$ and in the already picked hyperedges. This is doable since there are less than $t$ such vertices, and each such vertex intersects at most one of the $t$ hyperedges given by the $t$-fatness of $uv$.
\end{proof}

Let $G=G(\cG)$ denote the graph consisting of the $t$-fat edges with $t= (r-2)|E(F)|+|V(F)|$.

\begin{corollary}
    If $\cG$ is $F^{(r)+}$-free, then $G(\cG)$ is $F$-free.
\end{corollary}

Note that the above corollary also implies Theorem \ref{main}. Indeed, there are at most $\ex(n,K_r,G)$ hyperedges with each subedge being $t$-fat. The other hyperedges each contain a subedge that is not $t$-fat. Given such an edge $uv$ in the shadow graph, deleting $u$ and $v$ from the hyperedges containing both of them, we obtain an $(r-2)$-uniform hypergraph without a matching of size $t$. By a theorem of Erd\H os, \cite{erdo}, there are $O(n^{r-1})$ hyperedges for each of the $O(n^2)$ non-$t$-fat edges.

We also remark that essentially the same proof as in the above corollary gives that for any $i\le r$, the hypergraph having the $t$-fat $i$-sets as hyperedges is $F^{(i)+}$-free. 

\begin{thm}\label{structu}
Let $\chi(F)=k+1>r$ and $\cG$ be an $n$-vertex $F^{(r)+}$-free $r$-graph with $|E(\cG)|=\ex_r(n,F^{(r)+})$. Let $G=G(\cG)$. Then there is a $k$-partition of $V(G)$ to $V_1,\dots,V_k$, a constant $K=K(F)$ and a set $B$ of at most
$rK(\sigma(F)-1)$ vertices such that each member of $\cD(F)$ in $G$ inside a part shares at least two
vertices with $B$, every vertex of $B$ is adjacent to 
$\Omega(n)$ vertices in each part
and every vertex
of $V_i\setminus B$ is adjacent to $o(n)$ vertices in $V_i$ and all but $o(n)$ vertices in $V_j$ with $j\neq i$. 
\end{thm}

\begin{proof}
    Since $G$ contains $\cG_2$, there are $O(n^{r-1})$ hyperedges that are not copies of $K_r$ in $G$, using Theorem \ref{main}. Therefore, we are done unless $\cN(K_r,G)\ge \ex(n,K_r,F)-\zeta n^{r-1}$ for some $\zeta>0$, hence we can apply Theorem \ref{grafos}. Assume that there is a $u\in U$. Without loss of generality, $u\in V_1$, $u$ has $o(n)$ neighbors in $V_1$ and $\Omega(n)$ non-neighbors in $V_2$. Let us pick a set $S$ of $t$ vertices in $V_1\setminus (B\cup U)$. Then the common neighborhood $X$ of vertices in $S$ contains all but $o(n)$ vertices of $V(G)\setminus V_1$. Now we delete from $\cG$ all the hyperedges that contain $u$ and another vertex from $V_1$ or from $V(G)\setminus X$, and add all the hyperedges that contain $u$ and at most one vertex from each $V_i\cap X$, $i\ge 2$. 

    We will show that we added $\Omega(n^{r-1})$ hyperedges, deleted $o(n^{r-1})$ hyperedges and the resulting hypergraph $\cG'$ is $F^{(r)+}$-free, clearly contradicting the extremality of $\cG$.

    If $\cG'$ contains a copy of $F^{(r)+}$, let 
    $xy$ be an edge of the core $F$ such that the corresponding hyperedge of $F^{(r)+}$ contains $u$. Then this hyperedge does not meet $V_1\setminus \{u\}$, nor $V(G)\setminus X\setminus \{u\}$. Observe that if $x$ and $y$ are in different parts $V_i$, then they both must be in $X\setminus V_1$, thus $xy$ is a $t$-fat edge. If $x$ and $y$ are both in $V_i$, then the hyperedge was already in $\cG$. Therefore, these hyperedges and the hyperedges not containing $u$ are in $\cG$, and we can extend them to a copy of $F^{(r)+}$ in $\cG$ using Proposition \ref{trivi}, a contradiction.

    Consider the hyperedges in $\cG'$ that contain $u$, one of the $\Omega(n)$ vertices in $V_2$ that are not adjacent to $u$ in $G$, and at most one vertex from each other part. Clearly, there are $\Omega(n^{r-1})$ such hyperedges and each of them is new.

    Finally, consider the deleted hyperedges. Clearly there are $o(n^{r-1})$ hyperedges that contain $u$ and a vertex from $V(G)\setminus X\setminus V_1$, thus we can focus on those that contain $u$ and another vertex from $V_1$. Assume that $v\in V_1$ is in at least $\varepsilon n^{r-2}$ hyperedges together with $u$. Then deleting $u$ and $v$ from those hyperedges, we obtain an $(r-2)$-graph with at least $\varepsilon n^{r-2}$ hyperedges. This contains a matching of size $t$ by a theorem of Erd\H os \cite{erdo}, thus $uv$ is $t$-fat, hence there are at most $\varepsilon n$ such vertices $v$ by the definition of $U$. Those vertices are in at most $n^{r-2}$ hyperedges together with $u$. Therefore, altogether there are at most $2\varepsilon n^{r-1}$ hyperedges containing $u$ and another vertex from $V_1$, completing the proof.
\end{proof}

\section{Proofs}

Let us continue with Proposition \ref{corre}. The upper bound follows from Theorem \ref{main}. The lower bound is provided by 
$T_r'(n,k)$. Recall that it is obtained from $T_r(n,k)$ by taking $u,v\in V_1$ and adding as hyperedges each $r$-set that contains both $u$ and $v$, and each $r$-set that contains $v$ and $r-1$ vertices from parts other than $V_1$. The lower bound in Proposition \ref{corre} follows from the next proposition.

\begin{proposition}
    If $F^{(r)+}$ is contained in $\cH_r'(n,k)$, then either $F$ has chromatic number at most $k$, or $F$ has a color-critical edge.
\end{proposition}

\begin{proof} Let $\cH_r'(n,k)=\cH'(m)$.
    If $F^{(r)+}$ is in $\cH(m)$, then $F$ has chromatic number at most $k$. If $F^{(r)+}$ contains one additional hyperedge, then clearly $F$ contains an edge whose removal results in a $k$-colorable graph. Therefore, we can assume that two additional hyperedges are used. Both contain $v$, thus $v$ must be in the core of $F^{(r)+}$, with degree at least 2. More precisely, there are two vertices $w_1,w_2$ such that $vw_1$ and $vw_2$ are both in the core $F$ and the hyperedges $H_i$ corresponding to $vw_i$ are not in $\cH(m)$.
    
    Assume without loss of generality that given the core, we picked the hyperedges of $F^{(r)+}$ such that we use as many hyperedges of $\cH(m)$ as possible. If any of $w_1$ and $w_2$ are not in $V_1$, then the edge $vw_i$ is $t$-fat in $\cH(m)$ thus we could replace $H_i$ by a hyperedge of $\cH(m)$, a contradiction. Thus $w_1,w_2\in V_1$, thus $H_1$ and $H_2$ both contain $u$, a contradiction.
\end{proof}

Let us continue with the proof of Theorem \ref{exac} that we restate here for convenience.

\begin{thm*}\textbf{(i)} Let $F$ consist of $s$ components with chromatic number $k+1$, each with
a color-critical edge, and any number of components with chromatic number at most $k$. Then
$\ex_r(n,F^{(r)+})=|E(T_r(n,k,s-1))|$ for sufficiently large $n$.

\textbf{(ii)}
$\ex_r(n,B_{k+1,1}^{(r)+})=|E(\cH_r'(n,k))|$ for sufficiently large $n$.
\end{thm*}

We remark that $\ex(n,K_r,F)$ and $\ex(n,K_r,B_{k+1,1})$ were both determined in \cite{ge2}. We will apply  Theorem \ref{structu} for the extremal graphs, and then adapt the proofs from \cite{ge2} to our setting. We will also use the well-known fact (see, e.g., \cite{bde}) that if an $r$-graph does not contain $s$ pairwise disjoint hyperedges, then either there is a set $S$ of order at most $s-1$ such that each hyperedge intersects $S$, or there are $O(n^{r-2})$ hyperedges in the $r$-graph.

\begin{proof} The lower bounds are obvious in both cases.
To obtain the upper bounds, we apply Lemma \ref{structu} to the extremal $r$-graph $\cG$. 

Let us start with \textbf{(i)}. Assume first that there are $s$ disjoint hyperedges $H_1,\dots,H_s$ in $\cG$ such that each $H_i$ contains two vertices $u_i,v_i$ that belong to the same part $V_j$ and at least one of them is not in $B$. Let $F_1,\dots,F_s$ denote the components of $F$ with chromatic number $k+1$. We go through the edges $u_iv_i$ one by one
and we will extend them to $F_i$. This can be done using Proposition \ref{trivi}. We need to show a core of $F$ that uses only $t$-fat edges and does not use vertices of any $H_i$ except for $u_i$ and $v_i$.

Without loss of generality, $u_i,v_i\in V_1$. Let $A_1,\dots,A_k$ denote the parts
of the $k$-partite graph we obtain by deleting a color-critical edge from $F_i$, with $A_1$ containing the two endpoints of that edge. We will embed the vertices
in $A_j$ to $V_j$. We pick $|A_1|-2$ vertices from $V_1\setminus B$. Let $H=\cup_{i=1}^sH_i$. Recall that $u_i$, $v_i$ and the other vertices of $A_1$ have $\Omega(n)$ common neighbors in $V_2\setminus (B\cup H)$. We pick $|A_2|$ of them that avoid the vertices we already picked to be in the copy of $F$. Then the vertices we picked
to be in the copy of $F_i$ have 
$\Omega(n)$ common neighbors in $V_3\setminus (B\cup H)$, we pick $|A_3|$ of them that we
have not picked to be in our copy of $F$, and so on. We always have to avoid $O(1)$ already
picked vertices. This way we obtain an $F_i$ and ultimately $F_1,\dots,F_s$. Clearly we can pick
the remaining components in a similar way to obtain $F$, a contradiction.

If there are $s-|B|$ disjoint hyperedges in $\cG$ avoiding $B$, then we can find the desired disjoint hyperedges. Indeed, we go through the vertices of $B$ one by one and pick a hyperedge containing it. We pick one of its $\Omega(n)$ neighbors in its own part in $G$ that is not in any of the at 
most $s-1$ hyperedges picked earlier and not in $B$. This is a $t$-fat edge, thus there are $t$ hyperedges of $\cG$ containing it and disjoint otherwise. At least one of them avoids each of the hyperedges picked earlier.


Let $\cG'$ denote the subhypergraph of $\cG$ containing the edges that intersect less than $r$ parts and avoid $B$.
We obtained that the largest matching in $\cG'$ is of size at most $s-1-|B|$. 
By the fact mentioned before the proof, either there is a set $S$ of $s-1-|B|$ vertices such that each hyperedge in $\cG'$ contains at least one of them, or there are $O(n^{r-2})$ hyperedges in $\cG'$.

In the first case, each hyperedge of $\cG$ either intersects $r$ parts, or contains at least one vertex of $B\cup S$. This is a subhypergraph of the following $F^{(r)+}$-free construction: we take a complete $k$-partite $r$-graph, choose $s-1$ vertices and add each hyperedge containing one of those vertices. It is left to show that in this case the best is to take a balanced complete multipartite hypergraph outside those $s-1$ vertices. This holds by the following argument. Assume that, say, $|V_1|>|V_2|+1$. Then we move a vertex from $V_1$ to $V_2$. Each hyperedge intersects $|V_1\cup V_2|$ in at most two vertices. The number of hyperedges with intersection of order less than 2 does not change. The number of hyperedges with intersection of order 2 increases, since it is the number of ways we can pick the other $r-2$ vertices from the other parts (this does not change), times the number of ways to pick the two vertices from $V_1$ and $V_2$, which is the product of $|V_1|$ and $|V_2|$, thus increases.

In the second case, assume first that $|B|<s-1$. Then we can delete $\cG'$, fix a vertex $v$ outside $B$ and add all the hyperedges containing $v$. This way we deleted $O(n^{r-2})$ hyperedges and added $\Omega(n^{r-1})$ hyperedges. The resulting hypergraph $\cG''$ is $F^{(r)+}$-free since each hyperedge is between parts or contains a vertex from $B\cup\{v\}$. But $\cG''$ contains more hyperedges than $\cG$, a contradiction.

Finally, if $|B|=s-1$ then $\cG'$ must be empty, thus we are again in the situation where each hyperedge of $\cG$ either intersects $r$ parts, or contains at least one vertex of a set of at most $s-1$ vertices. We have already dealt with this case, completing the proof.

Finally, let us deal with the upper bound in \textbf{(ii)}. Let $\cG'$ denote again the subhypergraph of $\cG$ containing the hyperedges that intersect less than $r$ parts and avoid $B$. Note that $B$ is empty here. Assume first that there are two disjoint hyperedges in $\cG'$ such that each of them contains two vertices inside the same part. Then we can apply Proposition \ref{trivi} to obtain a copy of $B_{k+1,1}^{(r)+}$ in $\cG$, a contradiction.

Let us apply again the well-known fact mentioned before the proof (this special case is due to Hilton and Milner \cite{hm}). We obtain that either there are $O(n^{r-2})$ hyperedges in $\cG'$ (in which case $|E(\cG)|\le |E(T_r(n,k))|+O(n^{r-2})$, completing the proof), or there is a vertex $u$ contained in each hyperedge of $\cG'$. Assume without loss of generality that $u\in V_1$. 

Let $\cH$ denote the $(r-1)$-uniform hypergraph we obtain the following way. We take the hyperedges that contain at least two vertices in $V_1$, and remove $u$. Assume first that there are two disjoint hyperedges $H_1,H_2$ in $\cH$. Let $u_1\in (H_1\cap V_1)\setminus \{u\}$ and $u_2\in (H_2\cap V_1)\setminus \{u\}$. Then we take the edges $uu_1$ and $uu_2$. They are extended to a copy of $B_{k+1,1}$ with the edges of $G$, thus Proposition \ref{trivi} gives a contradiction. Applying the Hilton-Milner theorem again, we obtain that either $|\cH|=O(n^{r-3})$, or there is another vertex $v$ contained in each hyperedge of $\cH$. In the first case we are done, since each hyperedge of $\cG$ except the deleted ones are in $\cH'(m)$
for some $m$, and there are $\Omega(n^{r-2})$ additional hyperedges in $\cH'(m)$. In the
second case $|\cH|$ can be $\Omega(n^{r-2})$, but in this case $\cG$ is a subhypergraph of $\cH'(m)$, completing the proof.
\end{proof}

\section{Concluding remarks}

Another well-studied class of graph-based hypergraphs is called Berge hypergraphs \cite{gp1}. In expansions, the core graph $F$ is expanded vertex-disjointly; in Berge hypergraphs it is extended arbitrarily (thus we obtain a family of hypergraphs). The connection of Berge hypergraphs and generalized Tur\'an problems was established in \cite{gp2}, and it is even closer than in our case.

The Tur\'an number of Berge hypergraphs is subquadratic if the uniformity $r$ is large, a very different behavior from expansions, where $\Omega(n^{r-1})$ is a lower bound (except for stars). However, in the non-degenerate case they behave similarly, in particular for graphs of chromatic number more than $r$ and a color-critical edge we have equality. Our results (combined with results from \cite{ge3}) give the first non-degenerate examples ($B_{k+1,1}$ with $k\ge r$) where the Berge copies and expansions have different Tur\'an number.

There are other graph-based hypergraphs where the core $F$ is extended such a way that it is a Berge copy of $F$, and $F^{(r)+}$ fits the rules. One example is the \textit{induced Berge} or \textit{trace} (see \cite{fl}), when the vertices of the core cannot be used more than necessary, but any extension is possible outside (i.e., the trace of the hypergraph on $V(F)$ is $E(F)$). It would be interesting to extend our investigations to those hypergraphs and decide whether they can different from each other and expansions and Berge hypergraphs in the non-degenerate case.

\bigskip

\textbf{Funding}: Research supported by the National Research, Development and Innovation Office - NKFIH under the grants SNN 129364, FK 132060, and KKP-133819.






\begin{thebibliography}{99}

		
		\bibitem{ALS2016} N. Alon, C. Shikhelman. Many $T$ copies in $H$-free graphs. \textit{Journal of Combinatorial Theory, Series B}, \textbf{121}, 146--172, 2016.

\bibitem{bde}  B. Bollob\'as, D.E. Daykin, P. Erd\H os. Sets of independent edges of a hypergraph. \textit{Quart. J. Math.
Oxford Ser.} \textbf{27}, 25--32, 1976.

\bibitem{erdo} P. Erd\H os. A problem on independent r-tuples. \textit{Ann. Univ. Sci. Budapest.} \textbf{8}, 93--95, 1965.

		


\bibitem{ES1966} P. Erd\H os, M. Simonovits. A limit theorem in graph theory. \textit{Studia Sci.
Math. Hungar.} \textbf{1}, 51--57, 1966.

\bibitem{ES1946} P. Erd\H os, A. H. Stone. On the structure of linear graphs. \textit{Bulletin of the
American Mathematical Society} \textbf{52}, 1087--1091, 1946.

\bibitem{fl} Z. F\"uredi, R. Luo. Induced Turán problems and traces of hypergraphs. \textit{European Journal of Combinatorics}, 103692, 2023.




\bibitem{ge2} D. Gerbner. Some exact results for non-degenerate generalized Turán problems. \textit{ Electronic J. Comb.}, \textbf{30}(4), P4.39, 2023.

\bibitem{ge3} D. Gerbner. On non-degenerate Berge-Turán problems, \textit{Graphs and Combinatorics}, 40(2), 37, 2024.



		




\bibitem{gp1} D. Gerbner, C. Palmer, Extremal Results for Berge hypergraphs. \textit{SIAM Journal on Discrete Mathematics}, \textbf{31}, 2314--2327, 2017.

\bibitem{gp2} D. Gerbner, C. Palmer. Counting copies of a fixed subgraph in $ F $-free graphs. {\it European Journal of Combinatorics} {\bf 82}, 103001, 2019. 


       \bibitem{gp} D. Gerbner, B. Patk\'os, Extremal Finite Set Theory,
1st Edition, CRC Press, 2018.


\bibitem{hm} A. J. W. Hilton, E. C. Milner. Some intersection theorems for systems of finite
sets, \emph{Quart. J. Math. Oxford Ser.} \textbf{2}(18), 369--384, 1967.

\bibitem{kmv} A. Kostochka, D. Mubayi,  J. Verstra\"ete. Turán problems and shadows III: expansions of graphs. \textit{SIAM Journal on Discrete Mathematics}, \textbf{29}(2), 868--876, 2015.
        

\bibitem{mq} J. Ma, Y. Qiu, Some sharp results on the generalized Tur\'an numbers. \textit{European Journal of Combinatorics}, \textbf{84}, 103026, 2018.

\bibitem{mub} D. Mubayi. A hypergraph extension of Tur\'an's theorem.
{\it Journal of Combinatorial Theory, Series B}, 96, 122--134, 2006.

\bibitem{mubver} D. Mubayi, J. Verstra\"ete. A survey of Tur\'an problems for expansions. \textit{Recent Trends in Combinatorics}, 117--143, 2016. 

\bibitem{pttw} C. Palmer, M. Tait, C. Timmons, A.Z. Wagner. Tur\'an numbers for Berge-hypergraphs and related extremal problems. \textit{Discrete Mathematics}, \textbf{342}(6), 1553--1563, 2019.

\bibitem{pikhu} O. Pikhurko. Exact computation of the hypergraph Tur\'an function for expanded complete
2-graphs, \textit{Journal of Combinatorial Theory, Series B}, \textbf{103}(2) 220--225, 2013.

\bibitem{sim} M. Simonovits. A method for solving extremal problems in graph theory, stability
problems. \textit{Theory of Graphs, Proc. Colloq., Tihany, 1966, Academic Press, New
York}, 279--319, 1968.

\bibitem{simi} M. Simonovits. Extremal graph problems with symmetrical extremal graphs. Additional chromatic conditions, \textit{Discrete Math.} \textbf{7}, 349--376, 1974.

\bibitem{txy} Y. Tang, X. Xu, G. Yan, An Improved Error Term for Tur\'an Number of Expanded Non-degenerate 2-graphs, \textit{arXiv preprint}, arXiv:1904.00146


\bibitem{T}
P. Tur\'an. Egy gr\'afelm\'eleti sz\'els\H o\'ert\'ekfeladatr\'ol. \textit{Mat. Fiz. Lapok}, \textbf{48}, 436--452, 1941.
		

\end{thebibliography}
\end{document}